\documentclass[12pt]{article}
\usepackage[utf8]{inputenc}

\usepackage{tikz}

\usepackage{amsmath,amsthm,amscd,amssymb,eucal,mathrsfs}
\usepackage{amsfonts}
\usepackage{latexsym}
\usepackage{graphicx} 

\usepackage{multicol}
\usepackage{dsfont}
\usepackage[all]{xy}
\usepackage{multirow}
\usepackage{color}

\newtheorem{thm}{Theorem}[section]

\newtheorem{lem}[thm]{Lemma}
\newtheorem{cor}[thm]{Corollary}

\newtheorem{definition}[thm]{Definition}

\newtheorem{example}[thm]{Example}

\newtheorem{Remark}[thm]{Remark}

\newtheorem*{Satz*}{Satz}

\newtheorem{game}{Game}
\newtheorem{Lemma}[thm]{Lemma}

\newcommand{\mathset}[1]{{\left\{#1\right\}}}
\newcommand{\absolute}[1]{\left\lvert#1\right\rvert}
\newcommand{\norm}[1]{\left\|#1\right\|}

\DeclareMathOperator{\Spec}{Spec}

\title{Hearing Shapes via $p$-Adic Laplacians}
\author{Patrick Erik Bradley and \'Angel Mor\'an Ledezma}

\date{\today}

\begin{document}

\maketitle

\begin{abstract}
For a finite graph, a spectral curve is constructed as the zero set of a two-variate polynomial with integer coefficients coming from $p$-adic diffusion on the graph. It is shown that certain spectral curves can distinguish non-isomorphic pairs of isospectral graphs, and can even reconstruct the graph. This allows the graph reconstruction from the spectrum of the associated $p$-adic Laplacian operator. As an application to $p$-adic geometry, it is shown that the reduction graph of a Mumford curve and the product reduction graph of a $p$-adic analytic torus can be recovered from the spectrum of such operators. 
\end{abstract}

\section{Introduction}

The aim of spectral geometry is to describe the relationship between the geometry of certain objects like surfaces, or more general Riemannian manifolds, and the spectra of differential operators, like Laplacians, defined on them. In other words, as stated by M. Kac in \cite{Kac1966}, "Can one hear the shape of a drum?" Ideally, one would like to be able to recover the  geometric object, up to isometry, from the spectra of one or several naturally defined operators. Many counter-examples for Riemanian manifolds have appeared showing that isospectral but non-isometric manifolds exist, which gives a negative answer to the question. For example for the drum problem:
\begin{equation*}
    \begin{cases}
      \nabla u=-\lambda u\\
      u|_{\partial D}=0
    \end{cases}\,
\end{equation*}
Carolyn Gordon, David Webb, and Scott Wolpert in 1992, showed the existence of a pair of drums with different shapes but which are isospectral \cite{GWW1992}. On the other hand, information about the geometry of the object can be extracted from the spectrum. This kind of problems is known as ``inverse problems''. Many famous results towards this direction have been stablished, for example the famous Weyl asymptotic law \cite{Weyl1911}. These problems extend to objects other than Riemannian manifolds, like graphs, where for the adjacency matrix and the Laplacian, non-isomorphic isospectral graphs have been found. 
This problem has been intensively studied, cf.\ e.g.\ \cite{OB2012, MM2016, DH2003}. Recovering the structure of a graph from the spectrum of an operator may lead to an invariant to describe the topology of the graph. This could lead to new applications, like recovering the structure of a graph from a diffusion process, which has many potential applications, e.g.\ for topological access methods for spatial data \cite{JB2022}, to name only one. 
 
\smallskip
Prime numbers play a fundamental role in many mathematical theories and applications to sciences. From the realm of arithmetic as the fundamental blocks or "atoms" of integers, to applications in physics, from quantum physics  to the theory of complex disordered systems and geophysics, information processing, biology, and cognitive science, see \cite{OXOP2017} and the references therein. One powerful framework for the application of number theory in sciences is the so called $p$-adic analysis or more general ultrametric analysis \cite{VVZ1994,XKZ2018}.  An important example is given in the theory of disordered systems (spin glasses) where the $p$-adic structure is encoded in a Parisi matrix which arises by the intrinsic hierarchical structure inside the spin glasses \cite{PhysRevLett.52.1156}. This lead in the middle of the 80s to the idea of using ultrametric spaces to describe the state of complex systems. A central idea in physics of complex systems (like proteins) states that the dynamics on such systems is generated by a random walk (diffusion equation) in the corresponding energy landscape. By using interbasin kinetics methods, an energy landscape is approximated by an ultrametric space  and a function on this space describing the distribution of the activation barriers, see e.g.\ \cite{Kozyrev2011MethodsAA} and the references therein. Most of the applications towards this direction recquire  well-defined and natural pseuddifferential operators constructed on ultrametric  structures such as Non-Archimedean fields, where the Taibleson-Vladimirov operator plays a fundamental role for diffusion on the field $p$-adic numbers \cite{VVZ1994}.  Differential operators and spectral geometry on Riemannian manifolds have been extensively studied, nevertheless there is no comparable theory of Non-Archimedean spectral geometry of pseudodifferential operators over $p$-adic structures. Many other operators have been developed, some of them with the aim of applications, and others as generalisations to more general structures. For the former we have many classes of $p$-adic operators from the work of W.\ Zúñiga, Kozyrev, Khrennikov, where the relation of graph theory and $p$-adic integral and pseudodifferential operators is explicitly stated, see \cite{Zuniga2020,XKZ2018,SVKozyrev2003pAdicPO}, and the reference there in. 

For the latter one of the authors initiated the study of heat equations and integral operators on the Non-Archimedean kin of Riemannian surfaces i.e.\ Mumford Curves \cite{HeatMumf}. All those developments in the theory of pseudodifferential equations over Non-Archimedean spaces clearly deal (indirectly) with one of the main problems in spectral geometry, that is, direct problems in which a description of the eigenvalues is needed. 

\smallskip
In this article we initiate the study of inverse problems of spectral geometry in the Non-Archimedean framework. Moreover, a new invariant for an arbitrary combinatorial simple graph is introduced, showing that the spectra of certain $p$-adic operators defined on the graph lead to a complete characterisation of its isomorphisim class.  The question "Can you hear the shape of a graph?" has already been answered in different contexts. In \cite{Gutkin2001CanOH}, the question was posed in the context of quantum graphs, and was answered in the affirmative, that is, they showed  that the spectrum of the Schrödinger operator on a finite, metric graph determines uniquely the connectivity matrix and the bond lengths under certain conditions. In \cite{Lawn2021}, a new spectral invariant in quantum graphs has been introduced.  In \cite{Harrison2022CanOH}, it was proved that the spectral determinant of the Laplace operator on a ﬁnite connected metric graph determines the number of spanning trees under certain conditions. Understanding how the spectra of certain operators in general graphs  determine the geometry of a graph is an important task for applications like graph comparison in graph analytics. For example in \cite{Tsitsulin2018NetLSDHT},  the Network Laplacian Spectral Descriptor, a graph representation method that allows for straightforward comparisons of large graphs, is proposed. Moreover, our results are applied to $p$-adic structures like Mumford curves and $p$-adic analytic tori. Hence these results initiate the study of inverse problems in spectral geometry in the Non-Archimedean framework. \newline

Given a graph $G$ and a matrix $\Delta\in \mathbb{N}^{|G|\times |G|}$, we study a generalisation of a graph Laplacian $\Lambda_{G}^{\Delta}$ defined in $L^2(G\times K)$, where $K$ is a non-archimedean local field. The space $L^2(G\times K)$ can be decomposed as a direct sum of finite dimensional spaces 
of dimension $|G|$, this leads of the following representation of $\Lambda_{G}^{\Delta}$ ,
$$\Lambda_{G}^{\Delta}=\bigoplus_{G\mathcal{K}} L(G_r^{\Delta}),$$
where the matrices $L(G_r^{\Delta})$ are the Laplacian matrix of a weighted version of the graph, and for $r=1$, we have that $G_1=G$. Therefore, this operator can be understood as a direct sum of scaled replica of the original graph.  The spectrum of each copy belongs to a 
common plane algebraic curve $V(P_{G}^{\Delta})$ called the spectral curve of the graph.  For a suitable choice of $\Delta$ we prove that $P_G^{\Delta}$ is an invariant of the graph $G$.  This leads to a reconstruction theorem which enable us to reconstruct the graph through the spectra of the operator $\Lambda_{G}^{\Delta}$ (see Theorem $4.7$ and Corollary $4.9$). Finally using these results we are able to reconstruct the  reduction graph of a Mumford curve and the product graph coming from the reduction of a $p$-adic analytic torus using the spectrum of a $p$-adic Laplacian.

\section{Notation and Some Results from $p$-Adic Analysis}

In this section we review some results from $p$-adic analysis, for a complete exposition of the subject and proofs the reader may consult \cite{AKochubei}. \newline

Let $K$ be a Non-Archimedean local field. Let $|\cdot|_K$ denote the absolute value of the field $K$. Denote 
the local ring of $K$ 
by $\mathcal{O}_K=\{x\in K : |x|_K\leq 1\}$ and 
its maximal ideal by $\mathfrak{m}_K=\{x\in K:|x|_K<1\}$. Let $\chi$ be a fixed non-constant complex-valued additive character on $K$. We denote by $dx$ the Haar measure on the additive group of $K$, normalised such that the measure of $\mathcal{O}_K$ is equal to $1$. The Fourier transform of an absolute integrable complex-valued function $f\in L^1(K)$ 
will be written as 
\[
\mathscr{F}(f)(\xi)= \int_{K} \chi(\xi x)f(x)dx, \  \xi\in K. 
\]
If $\mathscr{F}(f)=\hat{f} \in L^1(K),$ we get the inversion formula
\[
f(x)=\int_{K} \chi(-x\xi)\hat{f}(\xi)d\xi. 
\]
Since the mapping $\mathscr{F}:L^1(K)\cap L^2(K)\rightarrow L^2(K)$ is an isometry, this mapping has an extension to an $L^2-$isometry from $L^2(K)$ into $L^2(K)$, where the inverse Fourier transform will be denoted as $\mathscr{F}^{-1}$.  
\newline

 Now we introduce the Vladimirov-Taibleson operator.  Let $\mathcal{D} \subset L^2(K)$  be its domain given by the set of those $f\in L^2(K)$, for which $|\xi|^{\alpha} \hat{u}(\xi)\in L^2(K)$. The Vladimirov operator $(\Delta^{\alpha},\mathcal{D})$,  $\alpha>0$, for $f\in \mathcal{D}$ is defined by 
\[
\Delta^{\alpha} f(x)= \mathscr{F}^{-1}_{\xi\mapsto x}(|\xi|_K^{\alpha} \mathscr{F}_{x\mapsto \xi} (f)(\xi))(x), \ x\in K. 
\]
The operator $\Delta^{\alpha}$ is an unbonded operator in $L^2(K)$, and since it is unitarily equivalent to the operator of multiplication by $|\xi|_K^{\alpha}$, it is self-adjoint, its spectrum consists of the eigenvalues $\lambda_{r}=q^{\alpha r}$, where $r\in \mathbb{Z}$ and $q$ is the cardinality of the residue field $\mathcal{O}_K /\mathfrak{m}_K$. Moreover we have the following result
\begin{thm}[Kozyrev]\label{KozyrevWavelet}
    There exist a complete orthonormal system of eigenfunctions of the operator $\Delta^{\alpha}$ of the form $\psi_{r,n}(x)\in L^2(K)$ , where $r\in \mathbb{Z}$ and $n\in \mathbb{N}$ such that 
\[\Delta^{\alpha} \psi_{r,n}(x)=q^{\alpha(1-r)} \psi_{r,n}.\]
\end{thm}

\begin{proof}
For the case of $\mathds{Q}_p$, cf.\ \cite{Kozyrev2002}.
The case of a general Non-Archimedean local field $K$,
cf.\ \cite{AKochubei}.
\end{proof}

Henceforth this basis of $L^2(K)$ from Theorem \ref{KozyrevWavelet} will be denoted by $\mathcal{K}$. 

\section{Spectral Curves for Diffusion Pairs}

In this section, we introduce the objects necessary for constructing the spectral curve of a so-called \emph{diffusion pair} which is actually nothing but a weighted graph, where the weights are integer powers of a fixed variable $Y$.
These objects are $p$-adic matrix-valued Laplacian operators reflecting the adjacency structure of a graph.

\subsection{$p$-Adic Laplacians for Graphs}\label{sec:Laplacian}

Let $G\subset K/O_K$ be a finite set. Then we have  isomorphisms
\[
L^2(G\times K)
\cong L^2(G)\otimes L^2(K)
\cong \bigoplus\limits_{a\in G}L^2(K_a)
\]
where $K_a$ is a copy of $K$ for each $a\in G$.
We define maps:
\begin{align}\label{Gsquare}
\xymatrix{
\bigoplus\limits_{a\in G}
L^2(K_a)\ar[r]^{H_G}&
\bigoplus\limits_{a\in G}
L^2(K_a) 
}
\end{align}
where we write
\[
L^2(K_a)=\bigoplus\limits_{\psi\in \mathcal{K}}
\mathds{C}\psi_a
\]
using the set $\mathcal{K}$ of Kozyrev wavelets on $K$, and
\[
\psi_a(x)=\psi(x).
\]
The map $H_G$ is given as follows:
\begin{align*}
H_G&\colon (u_g)_{g\in G}\mapsto
(f_g)_{g\in G},\;
f_g=\sum\limits_{a\in G}C_{ga}\Delta_{ga} u_a
\end{align*}
where $\Delta_{ga}$ is the Vladimirov operator
\[
\Delta_{ga}
\colon
L^2(K_a)\to L^2(K_g),
\psi_a\mapsto\Delta^{\alpha_{ga}}\psi_g
\]
where
\[
\Delta^{\alpha_{ga}}=\mathscr{F}^{-1}\absolute{\cdot}_K^{\alpha_{ga}}\mathscr{F}
\]
behaves like the usual Vladimirov operator, except for being applied to different copies of Kozyrev wavelets indexed by vertices of $G$. In particular, it simply multiplies the  indexed Kozyrev wavelet
$\psi_g$ by an integer power of $q^{\alpha_{ga}}$.

\smallskip
Notice that in the basis  of $L^2(G\times K)$
given by 
\[
G\mathcal{K}
=\mathset{\psi_g\colon g\in G,\;\psi\in\mathcal{K}}
\]
we can represent $H_G$ by the  
$\absolute{G}\times\absolute{G}$-matrix
\[
(C_{ag}\Delta_{ag})
\]
And the matrix $C=(C_{ag})$ can be viewed as an adjacency matrix of a simple graph with vertex set $G$. 

\smallskip
In order to obtain a graph Laplacian matrix, we consider instead of $H_G$ the operator
\[
\Lambda_G^\Delta\colon L^2(K)^{\absolute{G}}\to
L^2(K)^{\absolute{G}}
\]
represented by the matrix
\[
(L_{ab}^\Delta)_{a,b\in G}
\]
with
\[
L_{ab}^\Delta=
\begin{cases}
-C_{ab}\Delta_{ab},&a\neq b
\\
\sum\limits_{g\in G}C_{ag}\Delta_{ag},&a=b
\end{cases}
\]
Here, $\Delta=(\Delta_{ga})$ can be viewed as a a matrix in
$\mathds{N}^{\absolute{G}\times\absolute{G}}$ having entry $\alpha_{ga}$ whenever $ga$ represents an edge of the graph.

\smallskip
Later, we will show that there exist choices of diffusion parameters $\alpha_{ag}\in\mathds{N}$ such that
the spectrum of the operator
$\Lambda_G^\Delta$  determines the isomorphism class of
the combinatorial simple graph $G$.

\begin{definition}
The operator $\Lambda_G^\Delta$ is called the \emph{$p$-adic Laplacian associated with the diffusion pair $(G,\Delta)$}.
\end{definition}

\subsection{The Spectral Curve of a Diffusion Pair}

Let $(G,\Delta)$ be a diffusion pair. Recall that 
$L=L_1$ is the Laplacian of the graph $G$.
We begin with the following observation:
\begin{lem}
It holds true that
\[
\Spec(L)\subset\Spec(\Lambda_G^\Delta)
\]
as an inclusion of multi-sets.
\end{lem}

\begin{proof}
Observe first that in the basis $G\mathcal{K}$, the operator $\Lambda_G^\Delta$ is represented by an $\mathds{N}\times\mathds{N}$-matrix having a block-diagonal structure with blocks of size $\absolute{G}\times\absolute{G}$ after a suitable linear ordering of the basis.
Now, the non-zero entries of each block 
away from the diagonal
consist of Vladimirov eigenvalues. By Theorem \ref{KozyrevWavelet}, they are of the form
\[
q^{\alpha_{ag}(1-r)}
\]
with $r\in \mathds{Z}$. For $r=1$ we identify the Laplacian matrix $L$ as one of the blocks. Hence, the eigenvalues of $L$ are contained in the spectrum of $\Lambda_G^\Delta$. 
\end{proof}

Observe further that the block-diagonal structure found in the proof of the above Lemma is in fact a replication of Laplacian matrices for the same combinatorial graph structure on $G$, except that now the edge lengths 
are powers of $p$ of the form $p^{\alpha_{ab}(1-r)}$ for fixed $r\in \mathds{Z}$.
This means that there is a family  of graphs $G_r^\Delta$ parametrised by $r\in\mathds{Z}$ having the same combinatorial Laplacian matrix $L$. And in order to find all eigenvalues of $\Lambda_G^\Delta$, it is necessary and sufficent to find the Laplacian eigenvalues for each graph in the family
$G_r^\Delta$ with $r\in\mathds{Z}$.

\smallskip
We will now examine the characteristic polynomial of each graph Laplacian $L_r^\Delta$ associated with graph $G_r^\Delta$ from the family. Notice that
\[
L_1^\Delta=L_1=L,\quad G_1^\Delta=G_1=G,
\]
are independent of the diffusion parameters symbolised by $\Delta$. 
We also assume that the parameters $\alpha_{ab}$ are all pairwise different positive natural numbers.
The characteristic polynomial of $G_r^\Delta$
is
\[
P_r^\Delta(X)\in\mathds{Q}[X]
\]
and its degree is $\absolute{G}$. 
Again, we have
\[
P_1(X)=P_1^\Delta(X)
\]
is independent of $\Delta$, and coincides with the characteristic polynomial of the graph Laplacian $L$.

\smallskip
The coefficients of $P_r^\Delta$ are given by the Leibniz formula for the determinant as polynomials  with integer coefficients  in another variable $Y$ evaluated in $p^{1-r}$. 
Hence, we obtain a polynomial
\[
P_G^\Delta(X,Y)\in\mathds{Z}[X,Y]
\]
whose zero set in $\overline{\mathds{Q}}^2$ contains the spectrum of $\Lambda_G^\Delta$ as the first coordinate of some of its points. Here, we mean by $\overline{\mathds{Q}}$ the algebraic closure of $\mathds{Q}$.

\begin{definition}
The completion of the plane algebraic  curve
$V(P_G^\Delta)$ 
to a projective algebraic curve is called the \emph{spectral curve} of the pair $(G,\Delta)$. 
\end{definition}

\subsection{Recovering Spectral Curves}

Assume that we are given the $\Spec\Lambda_G^\Delta$ as a multi-set, where $\Lambda_G^\Delta$ is the $p$-adic Laplacian associated with diffusion pair $(G,\Delta)$. Assume also that the task is to recover the graph $G$ from that spectrum.
One way would be to try to recover the spectral polynomial
$P(X,Y)=P_{G}^\Delta(X,Y)$ and use the Reconstruction Theorem (Theorem \ref{reconstruct}) proved below.

\smallskip
In this situation, an algorithm which terminates in finite time cannot be expected, because each individual eigenvalue has to be associated with one of the graphs $G_r$ ($r\in\mathds{Z}$) in the family induced by the spectral pair.
But from a purely existential standpoint, we can say that there exists a classification of eigenvalues (including their multiplicities) such that each class is $\Spec(L_r)$ with $r\in\mathds{Z}$. Once this classification is made, then
each coefficient
\[
a_i(p^{1-r}),\quad i=1,\dots,n
\]
of the polynomial
\[
P(X,Y)=\sum\limits_{i=1}^na_i(Y)X^i
\]
with
$r\in\mathds{Z}$ can be calculated in each class of eigenvalues. All that is then needed, is for each $r\in\mathds{Z}$ the value of
\[
P(X,p^{1-r})
\]
in finitely many places $x_s$. Then interpolation yields the coefficients of $P(X,Y)$.

\begin{definition}
A set of pairs $(x_s,p^{1-r})$ with $s,r\in R\subset\mathds{N}$
is called a \emph{recovery datum}, if $R$ is a finite set and
$P(X,Y)$ can be interpolated after evaluating the polynomial in that set of pairs.
\end{definition}

\begin{thm}\label{existenceRecoveryDatum}
Let $(G,\Delta)$ be a diffusion pair.
Given $\Spec(L_r)$ as distinguished multi-sets for sufficiently but finitely many $r\in\mathds{Z}$,
it is possible to obtain recovery data for the spectral polynomial $P_{(G,\Delta})$ with a terminating algorithm.
\end{thm}

\begin{proof}
Since all  graphs $G_r$ are simple 
and have the same underlying combinatorial graph with $n$ vertices and
\[
\absolute{E}\le \frac12n(n-1)=:b_n
\]
edges, it follows that the number of places to interpolate 
\[
P(X,Y)=P_{G}^\Delta(X,Y)
\]
is bounded.
As $n$ is given as the size of each multi-set $\Spec(L_r)$,
it follows that $b_n$ such spectra are sufficient in order to reproduce the characteristic polynomials
\[
P(X,p^{1-r})=P_r(X)
\]
of the graphs $G_r$. Evaluating the $b_n$ polynomials $P_r(X)$ at $b_n$ places $x_s\in\mathds{R}$ yields pairs
$(x_s,p^{1-r})$ which form a recovery datum, as now interpolation of $P(X,Y)$ is possible. Together, this is an algorithm terminating after finitely many steps.
\end{proof}


The question is now, whether it is possible to extract distinguished multi-sets $\Spec(L_r)$ somehow by
 clustering spectral values.
 If it is allowed to  vary the prime number $p$, then this can be done in the following game:
 
 \begin{game}\label{game}
 Assume that you are allowed to choose diffusion parameters $\Delta$ once, and a prime number $p$ as many times as you wish.
 Then you will receive for each $p$ the multi-set $\Spec(\Lambda_G^\Delta)$ of an unknown simple finite connected graph. 
If you manage to recover the spectral curve for  the diffusion pair $(G,\Delta)$ from these spectra, then you win, otherwise you lose.
 \end{game}
 
 The following theorem states that there exist winning strategies for the Game \ref{game}:

\begin{thm}\label{winTheGame}
For the $p$-adic Laplacian associated with any diffusion pair, 
there exists a winning strategy in order to obtain distinguished multi-sets $\Spec(L_r)$ with $r\in\mathds{Z}$.
\end{thm}

\begin{proof}
In the case that the choice of diffusion parameters is 
to have them all equal to $1$ (when beloning to an edge, otherwise, it is zero), a winning strategy
for connected graphs is
to pick a large prime $p$.
In this case, we have
\[
L_r=p^{1-r}L
\]
where $L$ is the Laplacian of the graph $G$.

\smallskip
Since the non-zero part of the spectrum of $L$
lies inside a compact interval not containing $0$, as has been proven by \cite{Fiedler1973}, 
it then suffices to ask for higher and higher prime numbers until there are increasing gaps between clusters with relatively small inter-cluster distances between neigbouring points. With increasing $p$, this phenomenon becomes more and more clearly visible. 

\smallskip
If not all parameters are chosen equal to $1$, we restrict to $r\le 1$  and again vary the prime $p$.
From Matrix Perturbation Theory \cite{Baumgaertel1985},
we get that 
if $p$ is sufficiently large, then in the range $r\le 1$, again the inter-cluster distances of neighbouring eigenvalues will be smaller than the intra-cluster distances between neighbouring clusters. Hence, choosing $p$ sufficiently large, again removes overlaps between the clusters.

\smallskip
Since in both cases of diffusion parameter choices, the spectrum of $L$ remains fixed for any choice of varying the prime $p$,  the reference cluster for $r=1$
 can be extracted, and then the spectra $\Spec(L_r)$ for 
  more  values of $r\in\mathds{Z}$, or of $r\le 1$ in the second case. After having extracted sufficently many of these finite spectra, one can proceed to the interpolation method and compute a recovery datum, as  now the requirements for Theorem \ref{existenceRecoveryDatum} are met.
\end{proof}

\section{Spectral Curves Which are Separating}

In this section, we first separate pairs of non-isomrophic, but isospectral, graphs via spectral curves for suitable diffusion pairs using analytic matrix perturbation theory. After that, we prove for every finite graph the existence of a diffusion pair such that the graph can be reconstructed from the spectral polynomial. Although this generalises the first result, we believe that the matrix perturbation method is of general interest, nevertheless.

\begin{figure}[h]
\[
\xymatrix@C=10pt{
& &&7\ar@{-}[d]^2\ar@{-}[drr]&\\
&1\ar@{-}[urr]\ar@{-}[dl]&&8\ar@{-}[rr]\ar@{-}[ll]&&6\\
2\ar@{-}[drr]&&&&&&5\ar@{-}[ul]\\
&&3\ar@{-}[rr]&&4\ar@{-}[urr]
}
\qquad
\xymatrix{
&1\ar@{-}[r]\ar@{-}[ddl]&8\ar@{-}[dr]\ar@{-}[ddr]_2&\\
2\ar@{-}[ur]\ar@{-}[d]&&&7\ar@{-}[d]\\
3\ar@{-}[dr]&&&6\ar@{-}[dl]\\
&4\ar@{-}[r]&5
}
\]
\caption{A pair of non-isomorphic, but isospectral graphs whose first Betti number is  $3$.}\label{isospectralPair}
\end{figure}
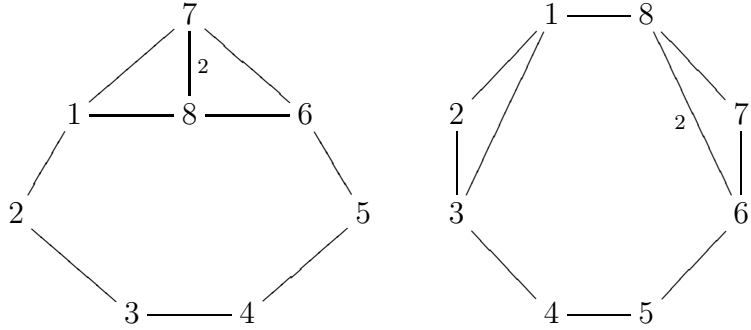

\begin{example}

According to \cite{MM2016}, the two graphs in Figure \ref{isospectralPair}
are isospectral. Their first Betti number equals $3$. This is the smallest example of 
an isospectral pair of non-isomorphic simple graphs without bridges.
The label "2" on an edge indicates a diffusion parameter value of $2$, i.e.\ an edge weight $Y^2$. Unlabelled edges have diffusion parameter value $1$, i.e.\ edge weight $Y$.
Their respective spectral polynomials $P_1$ for the left, and $P_2$ for the right graph of Figure \ref{isospectralPair} are:
\begin{align*}
P_1(X,Y)&=
\det\left(\begin{smallmatrix}
X-3Y&Y&0&0&0&0&Y&Y\\
Y&X-2Y&Y&0&0&0&0&0\\
0&Y&X-2Y&Y&0&0&0&0\\
0&0&Y&X-2Y&Y&0&0&0\\
0&0&0&Y&X-2Y&Y&0&0\\
0&0&0&0&Y&X-3Y&Y&Y\\
Y&0&0&0&0&Y&X-(2Y+Y^2)&Y^2\\
Y&0&0&0&0&Y&Y^2&X-(2Y+Y^2)
\end{smallmatrix}
\right)
\\[3mm]
P_2(X,Y)&=
\det\left(
\begin{smallmatrix}
X-3Y&Y&Y&0&0&0&0&Y\\
Y&X-2Y&Y&0&0&0&0&0\\
Y&Y&X-3Y&Y&0&0&0&0\\
0&0&Y&X-2Y&Y&0&0&0\\
0&0&0&Y&X-2Y&Y&0&0\\
0&0&0&0&Y&X-(2Y+Y^2)&Y&Y^2\\
0&0&0&0&0&Y&X-2Y&Y\\
Y&0&0&0&0&Y^2&Y&X-(2Y+Y^2)
\end{smallmatrix}
\right)
\end{align*}
Their tangent cones $T_1$ of $P_1$ and $T_2$ of $P_2$ are:
\begin{align*}
T_1(X,Y)&=
\det
\left(
\begin{smallmatrix}
X-3Y&Y&0&0&0&0&Y&Y\\
Y&X-2Y&Y&0&0&0&0&0\\
0&Y&X-2Y&Y&0&0&0&0\\
0&0&Y&X-2Y&Y&0&0&0\\
0&0&0&Y&X-2Y&Y&0&0\\
0&0&0&0&Y&X-3Y&Y&Y\\
Y&0&0&0&0&Y&X-2Y&0\\
Y&0&0&0&0&Y&0&X-2Y
\end{smallmatrix}
\right)
\\[3mm]
T_2(X,Y)&=\det
\left(
\begin{smallmatrix}
X-3Y&Y&Y&0&0&0&0&Y\\
Y&X-2Y&Y&0&0&0&0&0\\
Y&Y&X-3Y&Y&0&0&0&0\\
0&0&Y&X-2Y&Y&0&0&0\\
0&0&0&Y&X-2Y&Y&0&0\\
0&0&0&0&Y&X-2Y&Y&0\\
0&0&0&0&0&Y&X-2Y&Y\\
Y&0&0&0&0&0&Y&X-2Y
\end{smallmatrix}
\right)
\end{align*}
These are  spectral polynomials of the two bridgeless graphs of genus two, where each edge has the same variable $Y$, shown in Figure \ref{genus2pair}.
According to \cite[Thm.\ 3.1]{MM2016}, these graphs are not isospectral, because they are not isomorphic.
It follows that the two tangent cones 
$T_1$ and $T_2$
are not equal. 

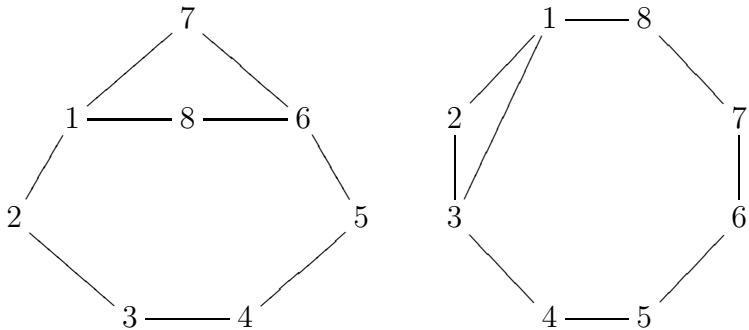
\begin{figure}[h]
\[
\xymatrix@C=10pt{
& &&7\ar@{-}[drr]&\\
&1\ar@{-}[urr]\ar@{-}[dl]&&8\ar@{-}[rr]\ar@{-}[ll]&&6\\
2\ar@{-}[drr]&&&&&&5\ar@{-}[ul]\\
&&3\ar@{-}[rr]&&4\ar@{-}[urr]
}
\qquad
\xymatrix{
&1\ar@{-}[r]\ar@{-}[ddl]&8\ar@{-}[dr]&\\
2\ar@{-}[ur]\ar@{-}[d]&&&7\ar@{-}[d]\\
3\ar@{-}[dr]&&&6\ar@{-}[dl]\\
&4\ar@{-}[r]&5
}
\]
\caption{A pair of non-isomorphic bridgeless graphs whose first Betti number is $2$. According to \cite{MM2016}, it follows that they are not isospectral.}\label{genus2pair}
\end{figure}

\smallskip
We saw that replacing one certain edge weight in each graph by $Y^2$ leads to two polynomials $P_1(X,Y)$ and $P_2(X,Y)$ 
which are not the same. So, in this case, the isospectral pair is separated by these two polynomials.
It only happens that
\[
P_1(X,1)=P_2(X,1)
\]
resulting in two identical characteristic polynomials for the two non-isomorphic graphs. This example motivates the remainder of this article.
\end{example}

\subsection{Separating Isospectral Pairs via Matrix Perturbation}

An introduction to matrix perturbation theory can be found in \cite{Baumgaertel1985}. We will use this method in order to construct distinct spectral polynomials for non-isomorphic, but isospectral graphs. All our calculations are explicit and most of them reproduced here, even if many can also be found in that bibliographic reference in a more general setting.

\smallskip
Let $k\neq\ell$ be natural numbers in $\mathset{1,\dots,n}$. We define the matrix
\[
U(k,\ell)=(u_{ij})
\]
with
\[
u_{ij}=
\begin{cases}
1,&i=j=k,\;\text{or}\;i=k=\ell
\\
-1,&(i,j)=(k,\ell),\;\text{or}\;(i,j)=(\ell,k)
\\
0,&\text{otherwise}
\end{cases}
\]
This is the Laplacian of the graph on $n$ vertices having 
precisely one undirected edge $(k,\ell)$, since $U(k,\ell)=U(\ell,k)$.

\smallskip
Let $v=(v_1,\dots,v_n)\in\mathds{R}^n$. Then
\[
v^\top U(k,\ell)v
=(v_k-v_\ell)^2
\]
as can be verified by a simple calculation.

\smallskip
Now, let $E$ be a symmetric set of pairs $(i,j)$ of numbers $1,\dots,n$ with $i\neq j$. Then
\[
U(E):=\sum\limits_{(k,\ell)\in E}
U(k,\ell)
\]
and we have
\[
v^\top U(E) v
=\sum\limits_{(k,\ell)\in E}(v_k-v_\ell)^2
\]
We can now define 
\[
\norm{v}_E^2=v^\top U(E) v
\]
The function $\norm{\cdot}_E$ is a semi-norm on $\mathds{R}^n$, and we have
\[
\norm{v}_E^2=0\quad\Leftrightarrow\quad
\forall\; (i,j)\in E\colon v_i=v_j
\]

\begin{Lemma}\label{separateEigen}
Let $L$ be the Laplacian of a simple graph $G$ with $n$ vertices, and let
$E,E'$ be two disjoint sets of pairs   from
$\mathset{1,\dots,n}$. Then there exists an eigenvector $v$ of $L$ such that
\[
\norm{v}_{E}^2\neq\norm{v}_{E'}^2
\]
\end{Lemma}

\begin{proof}
This follows from the fact that $L$ is symmetric, i.e.\ from the Spectral Theorem for symmetric real-valued matrices. Namely, if
for all eigenvectors $v$ of $L$ we had 
\[
\norm{v}_E=\norm{v}_{E'}
\]
then it would be impossible to generate via linear combination a vector $x\in\mathds{R}^n$ having 
$\norm{x}_E\neq\norm{x}_{E'}$, because both seminorms scale identically when an eigenvector is multiplied with a scalar.
\end{proof}

An equivalence of diffusion pairs $(G_1,\Delta_1)$ and $(G_2,\Delta_2)$ where $G_i$ are graphs on the same vertex set $V$, is given by a bijection $E_1\to E_2$ between the edge sets of $G_1$ and $G_2$ which takes a weighted edge to an edge having the same weight.
If two graphs $G_1,G_2$
are isospectral, then it is possible to define diffusion parameters $\Delta_1,\Delta_2$ such that there is an equivalence of diffusion pairs $(G_1,\Delta_1)\sim(G,\Delta_2)$. The following theorem shows that this already allows to distinguish non-isomorphic graphs, if suitable choices are taken.

\begin{thm}
Assume that $G_1,G_2$ is a pair of non-isomorphic, but isospectral graphs. Then there exist equivalent diffusion pairs
$(G_1,\Delta_1)$, $(G_2,\Delta_2)$ such that
\[
P_{G_1}^{\Delta_1}(X,Y)\neq P_{G_2}^{\Delta_2}(X,Y) 
\]
as bivariate polynomials.
\end{thm}

\begin{proof}
Denote the edge set of graph $G_i$ as $E_i$ for $i=1,2$, 
and 
let $C=E_1\cap E_2$, and $C_1=E_1\setminus E_2$, 
$C_2=E_2\setminus E_1$.
Since $G_1$ and $G_2$ are isospectral, it follows that
$E_1$ and $E_2$ have the same number of elements.

\smallskip
We are interested in the spectrum of the matrices
\[
L_{\varepsilon,i}=U(C)+\varepsilon U(C_i)
\]
for $i=1,2$, and with $\varepsilon >0$. In first order w.r.t.\ the parameter $\varepsilon$, we have
\begin{align}\label{varyLambda}
\lambda_{\varepsilon,i}^{(1)}=\lambda+\varepsilon\norm{v}_{C_i}^2
+\text{terms of higher order in $\varepsilon$}
\end{align}
where $v\in\mathds{R}^n$ is an eigenvector of $U(C)$ associated with eigenvalue $\lambda$, and $\lambda_{\varepsilon,i}^{(1)}$ approximates in first order an eigenvalue $\lambda_{\varepsilon,i}$ of $L_{\varepsilon,i}$.
According to Lemma \ref{separateEigen}, one can find an eigenvector $v$ of $U(C)$ such that
\[
\norm{v}_{C_1}\neq\norm{v}_{C_2}
\]
This implies that the first order approximations of the
eigenvalues $\lambda_{\varepsilon,i}$ are different.
It follows that the eigenvalues $\lambda_{\varepsilon,1}$ and $\lambda_{\varepsilon,2}$ are different for the range of $\varepsilon>0$ in which an analytic expansion in $\varepsilon$ is possible and $\varepsilon>0$ is sufficiently small.

\smallskip
Now, we have varied one eigenvalue of $U(C)$ analytically in two different ways. The other eigenvalues of $U(C)$ also vary analytically as in (\ref{varyLambda}), except that possibly it could be that some eigenvalues corresponding to $L_{\varepsilon,1}$ and $L_{\varepsilon,2}$ are equal
in this case. In any case, the spectrum of $U(C)=L_{\varepsilon,0}$ is varied continuously in two different manners. Hence, for $\varepsilon>0$ small, the discrete subsets $\Spec(L_{\varepsilon,1})$ and $\Spec(L_{\varepsilon,2})$ 
of $\mathds{R}$ are different. This implies that the corresponding characteristic polynomials 
\[
P_i(X,\varepsilon),\quad i=1,2
\]
are different. But these are evaluations of the spectral polynomials
\[
P_i(X,Y)
\]
evaluated at $Y=\varepsilon$ for all $\varepsilon>0$ sufficiently small, and where these polynomials are given
for edge weight maps:
\[
\Delta_i\colon C\cup E_i\to\mathds{N},\;
e\mapsto
\begin{cases}
1,&e\in C\\
0,&e\in E_i
\end{cases}
\]
It follows that these two spectral polynomials are different.

\smallskip
In order to obtain positive diffusion constants, now look at the spectrum of the matrix
\[
\varepsilon L_{\varepsilon,i}=\varepsilon U(C)+\varepsilon^2 U(C_i)
\]
which amounts to taking the diffusion parameters as
\[
\Delta_i\colon C\cup E_i,
\;e\mapsto
\begin{cases}
2,&e\in E_i\\
1,&e\in C
\end{cases}
\]
Hence, there exists a bijection $E_1\to E_2$ between sets of weighted edges, i.e.\ an equivalence of diffusion pairs $(G_1,\Delta_1)\sim(G_2,\Delta_2)$,
such that $P_{G_1}^{\Delta_1}(X,Y)\neq P_{G_2}^{\Delta_2}(X,Y)$ as asserted.
\end{proof}

\subsection{A Reconstruction Theorem}

Let $G=(V(G),E(G))$ be a graph with $n$ vertices.
We begin with the following observation:

\begin{thm}[Kel'mans, 1967]\label{Kelmans}
Let 
\[
P(X)=X^n-\sum\limits_{i=1}^{n-1}(-1)^ic_iX^{n-i}
\]
be the characteristic polynomial of the graph $G$. Then
\[
c_i=\sum\limits_{S\subset V\atop \absolute{S}=n-i}
T(G_S)
\]
where $T(H)$ is the number of spanning trees of graph $H$, and
$G_S$ is the quotient graph obtained by identifying all vertices in $S$ with a single vertex.
\end{thm}

\begin{proof}
\cite{Kelmans1967,KC1974}
\end{proof}

We recover these quantities from the spectral polynomial  
$P(X,Y)$ as follows:
\[
c_i(G)=a_{n-i}(1)
\]
for $i=1,\dots, n-1$.

\smallskip
We define
\[
\mathcal{F}(G)=
\mathset{F\colon \text{$F$ is a forest with $V(F)=V(G)$ and $E(F)\subset E(G)$}}
\]
and call $\mathcal{F}(G)$ the \emph{spanning forest set of $G$}. The following subsets of $\mathcal{F}(G)$ are of interest:
\[
\mathcal{F}^i(G)=\mathset{F\in\mathcal{F}(G)\colon
b_0(F)=i}
\]
This allows us to formulate a generalisation of Theorem \ref{Kelmans}  which is also a known result, but formulated here in the guise of spectral curves:

\begin{thm}[Buslov, 2014]
Let $(G,\Delta)$ be a diffusion pair.
Then the coefficients of the spectral polynomial  
\[
P_{G}^{\Delta}(X,Y)
=\sum\limits_{i=1}^na_i(Y)X^i
\]
of a diffusion pair $(G,\Delta)$
are given as
\[
a_i(Y)=(-1)^{n-i}
\sum\limits_{F\in\mathcal{F}^i(G)} \pi_F,\quad
\pi_F=\prod\limits_{e\in E(F)}Y^{\alpha_e}
\]
for $i=1,\dots,n$.
\end{thm}

\begin{proof}
\cite[Thm.\ 2]{Buslov2016}.
\end{proof}

Let $G'$ be another graph. An isomorphism
\[
\mathcal{F}(G)\cong\mathcal{F}(G')
\]
between spanning forest sets
is given by a bijection $f\colon \mathcal{F}(G)\to\mathcal{F}(G')$ and
an ismorphism $F\cong F'$ with
$F\in\mathcal{F}(G)$ and $F'=f(F)$ for all $F\in\mathcal{F}(G)$.

\begin{Lemma}\label{spanningForestUnique}
Let $G,G'$ be two graphs on $n$ vertices. Then
$G$ and $G'$ are isomorphic if and only if the strict forest subsets $\mathcal{F}(G)$ and $\mathcal{F}(G')$ are isomorphic.
\end{Lemma}

\begin{proof}
If the graphs are isomorphic, then clearly their strict forest sets are isomorphic.

\smallskip
Assume now that $G=T$ and $G'=T'$ are trees.
Since trees are their own spanning forests,  we  clearly must have that $T\cong T'$.

\smallskip
Now, assume that $G$ is not a tree. If $G$ is a forest,
then we can apply the result for trees to each individual connected component of $G$. So, we may assume that $b_1(G)>0$, and that $G$ is connected.
If $\mathcal{F}(G)\cong \mathcal{F}(G')$, then w.l.o.g.\ we may assume that these two sets are equal, and that the vertex sets of the two graphs coincide. 
Let $T$ be a spanning tree of $G$. Then
\[
\mathcal{F}(T)\subset\mathcal{F}(G)=\mathcal{F}(G')
\]
Hence, by symmetry, each spanning tree of $G$ is a spanning tree of $G'$ and vice versa. This implies that
$G=G'$, as otherwise there is an edge of $G$ not in $G'$. But then a spanning tree of $G$ containing that edge is not a spanning tree of $G'$, a contradiction.
Hence, $G\cong G'$.
\end{proof}

\begin{thm}[Reconstruction Theorem]\label{reconstruct}
For any finite graph $G$, there exists a $p$-adic Laplacian given by diffusion parameters $\Delta$ such that the diffusion pair $(G,\Delta)$ has a spectral polynomial $P(X,Y)$ such that for any diffusion pair $(G',\Delta)$ its spectral polynomial $P'(X,Y)$ satisfies:
\[
P(X,Y)=P'(X,Y)\quad\Leftrightarrow\quad (G,\Delta)\cong (G',\Delta)
\]
In other words, the isomorphism class of $G$ is the unique family of graphs having spectral polynomial $P(X,Y)$.
\end{thm}

\begin{proof}
We need only prove that if $P(X,Y)=P'(X,Y)$, then $(G,\Delta)\cong (G',\Delta')$. 

\smallskip
Let
\[
\Delta\colon E(G)\to\mathset{\alpha_1,\dots,\alpha_{\absolute{E(G)}}
},
e\mapsto 
\alpha_e
\]
be a bijection. W.l.o.g. we may assume that the edges of $G$ are numbered as $\alpha_e$, and that $\Delta$ is the identity map, so that  we may
write
$Y^e$ instead of $Y^{\alpha_e}$ in our polynomials.
We further make the following assumption:

\begin{align}\label{assumption}
&\text{We assume that no edge label  equals the finite sum of}
\\\nonumber
&\text{ any other edge labels (they are all positive integers).}
\end{align}

\smallskip
Let $I\subset E(G)$. Then we define
\begin{align*}
\mathcal{F}^i_I(G)&=\mathset{F\in\mathcal{F}^i(G)\colon
E(F)= I}
\\
\mathcal{F}^i(I)&=\mathset{F\in\mathcal{F}^i(G)\colon E(F)\subseteq I}
\end{align*}
Then $\mathcal{F}^i(G)$ is the disjoint union of all the sets $\mathcal{F}^i_I(G)$ for $I\subset E(G)$. Also, $\mathcal{F}^i_I(G)$ is either empty or consists of precisely one forest.

\smallskip
Let $E(G)=:E$, and let
$H$ be a spanning subgraph of $G$.
Then 
\[
a_{H,i}(Y)=(-1)^{n-i}
\sum\limits_{F\in\mathcal{F}^i(H)}\prod\limits_{e\in E(F)}Y^e
\]
and for any spanning tree $T$ of $G$, we have
\begin{align}\label{polynomialDecomp}
a_i(Y)=a_{T,i}(Y)+ a_{G\setminus T,i}(Y)
\end{align}
where $G\setminus T$ is obtained from $G$ by removing the edges of $T$.

\smallskip
Now, given a spectral polynomial $P(X,Y)$, the polynomial
\[
a_1(Y)=\sum\limits_{k=1}^Ma_{1k}Y^k
\]
with $M>>0$ recovers the set of spanning trees as follows:
a monomial 
$a_{1k}Y^k$
means that there are $\absolute{a_{1k}}\in\mathds{N}$
spanning trees  whose total sum of edge labels equals $k$.
Because of our assumption (\ref{assumption}), we have
that 
\[
\absolute{a_{1k}}\le 1
\]
for all $k=1,2,\dots$. 
Hence, each non-zero coefficient of $a_1(Y)$ encodes precisely one spanning tree of $G$.
Also, the two parts of the decomposition (\ref{polynomialDecomp}) have no monomials in common.

\smallskip
In order to recover the diffusion constants, we look at
\[
a_{n-1}(Y)
\]
Again, assumption (\ref{assumption}) ensures that all coefficients are either $1$ or zero. So, the exponents of the non-zero monomials retrieve all the distinct labelled edges of $G$.

\smallskip
We can go now further to extract for every spanning tree with a given set of edge labels, the set of all spanning forests, thereby knowing their numbers of connected components.
Beginning with all pairs of distinct edges, we can also  extract the sets of edges in each connected component of any spanning forest. The assumption (\ref{assumption}) makes this possible. In this way, a unique spanning tree is constructed.
Doing this for all spanning trees,  we obtain  a unique set of 
spanning forests $\mathcal{F}(G)$ for some graph $G$.
By Lemma \ref{spanningForestUnique}, the isomorphism class of $G$ is now uniquely determined.
\end{proof}

\begin{Remark}
We remark that the spectral polynomial for any diffusion pair satisfying (\ref{assumption}) has coefficients $0,\pm 1$, as has been seen in the proof of the Reconstruction Theorem 
(Theorem \ref{reconstruct}).
\end{Remark}

\begin{cor}\label{game2graph}
Playing Game \ref{game}
with choosing
diffusion parameters which satisfy (\ref{assumption}),  
allows you to reconstruct the unknown graph $G$. 
\end{cor}

\begin{proof}
This is an immediate consequence of Theorem \ref{reconstruct},
because Game \ref{game} has a winning strategy for any choice of diffusion parameters according to Theorem \ref{winTheGame}.
\end{proof}

\begin{Remark}
Notice that we are not solving the graph isomorphism problem in polynomial time, as the computation of the coefficients of the spectral polynomial can be expected to be far too time-consuming.
\end{Remark}

\section{Hearing Shapes of $p$-adic geometric objects}

Corollary \ref{game2graph} can also be applied 
to objects of $p$-adic geometry which have an underlying graph sructure. The first kind of objects consists of Mumford curves which have reduction graphs whose first Betti number equals the genus of the curve. The second kind are analytic tori which after a base change look like products of Mumford curves of genus $1$, so-called \emph{Tate curves}.
In order to be able to do this, we will construct embeddings of the sets of $K$-rational points of these objects into $K$.
The details of this procedure are presented in the  following two subsections.

\subsection{Hearing the Shape of a Mumford Curve}

Mumford curves are explained in some detail in \cite[Ch.\ 5]{FP2004}. However, we will not need to much of their construction. All we need is that they are projective algebraic curves admitting a finite cover by holed disks.
Certain types of coverings by holed disks in $K$, called \emph{verticial coverings} are introduced in \cite{HeatMumf}. These produce certain types of reduction graphs, also called verticial.

\smallskip
Let $X$ be a Mumford curve. In \cite{HeatMumf}, the concept of verticial covering of the set $X(K)$ of its $K$-rational points was developped. Let such a covering be given, and let $G$ be the corresponding verticial reduction graph. It is a connected graph whose vertices all  have degree at least $2$, and its first Betti number equals the genus $g$ of the curve.
We require Mumford curves to have positive genus.

\smallskip
The verticial covering of $X(K)$ consists of holed disks which are in one-to-one correspondence with the vertices of $G$.
The edges of $G$ correspond to annuli (as rigid analytic spaces)
with minimal positive thickness, i.e.\ they do not contain $K$-rational points. These annuli connect two otherwise disjoint holed disks in $X(K)$ without introducing extra $K$-rational points.

\smallskip
As a $p$-adic manifold, $X(K)$ is simply a disjoint union of finitely many holed disks. This compact manifold can be embedded into $K$ as a closed-open subset in such a way that each patch of the embedded verticial covering of $X(K)$ contains a distinct point representing a class in $K/O_K$.
Now, we are in  the setting of Section \ref{sec:Laplacian}, 
and have a $p$-adic Laplacian operator acting on the space
$L^2(K)^{\absolute{G}}$.

\begin{cor}
Playing Game \ref{game} with diffusion parameters satisfying (\ref{assumption}) allows to reconstruct a verticial reduction graph of a Mumford curve.
\end{cor}

\begin{proof}
This is immediate from the construction above and Corrolary \ref{game2graph}.
\end{proof}

\subsection{Hearing the Shape of a $p$-Adic Analytic Torus}

The theory of $p$-adic analytic tori is outlined in \cite[Ch.\ 6]{FP2004}. We will collect the data in what follows, and then proceed with the application.

\smallskip
Let $A$ be an analytic torus
\[
A=\mathds{G}_m/\Lambda
\]
with multiplicative lattice $\Lambda$ generated by a basis
\[
\tilde{q}_i=q_1^{\alpha_{i1}}\cdots q_g^{\alpha_{ig}}
\]
for $i=1\dots,g$ with
\[
q_i=q e_i
\]
where $q=p^f$, $e_i$ is the unit vector of the $i$-th component in the product
space $K^\times\times\cdots\times K^\times$, and $\alpha_{ij}\in\mathds{Z}$.

\smallskip
The lattice basis yields a decomposition
\[
K^g=\bigoplus\limits_{i=1}^g
K\tilde{q}_i
\]
On each component $K\tilde{q}_i$, we can define an ultrametric norm as follows:
the generator $\tilde{q}_i$ is determined by an integer vector 
\[
\alpha_i=(\alpha_{i1},\dots,\alpha_{ig})
\in\mathds{Z}^g
\]
Its associated \emph{primitive vector}
is a vector
\[
\lambda_i\in\mathds{N}^g
\]
such that
\[
\alpha_i=k\cdot\lambda_i
\]
with $k\in\mathds{Z}$, and $\absolute{k}$ is maximal with this property.
If we write
\[
\tilde{q}=q_1\cdots q_g
\]
and 
\[
\tilde{q}^\beta=
q_1^{\beta_1}\cdots q_g^{\beta_g}
\]
for 
\[
\beta=(\beta_1,\dots,\beta_g)\in\mathds{Z}^g
\]
Then we have
\[
\tilde{q}_i=\tilde{q}^{\alpha_i}
\]
Its associated \emph{primitive generator} of the line $K\tilde{q}_i$ is defined as
\[
\tilde{b}_i:=\tilde{q}^{\lambda_i}
\]
where $\lambda_i$ is the primitive vector associated with $\alpha_i$.

\begin{definition}
The ultrametric norm associated with the line $K\tilde{q}_i$ is defined as
\[
\norm{x}_{i}
=
\norm{\tilde{b}_i}_{K^g}^{-\log_q\absolute{\lambda}_K}
\]
where $x=\lambda \tilde{b}_i\in K\tilde{q}_i$ with 
primitive generator $\tilde{b}_i$, and $\lambda\in K$.
The Haar measure $\mu_i$ on 
$K\tilde{q}_i$ is normalised such that
the unit ball w.r.t.\ $\norm{\cdot}_i$ has measure one.
\end{definition}

The element $\tilde{q}_i$ defines a Tate curve, whowe reduction graph we assume to be simplicial, as follows: as we have
\[
\alpha_i=n_i\lambda_i
\]
with natural $n_i>2$,
it follows that
\[
\norm{\tilde{q}_i}_i=\norm{\tilde{b}_i}^{n_i}
=(q^{-c_i})^{n_i}
\]
with integer $c_i>0$.
Hence, the component Tate curve here is
\[
T_i=(K^\times \tilde{b}_i)/\langle\tilde{q}_i\rangle
\cong
K^\times/\left\langle (q^{c_i})^{n_i}\right\rangle
\]
However, there is a difference in the reduction graph structures on both sides of the isomorphism: a verticial covering of the curve on the right has
$c_in_i$ vertices, whereas anyone on the left has $n_i$ vertices.

\smallskip
We have a decomposition
\[
L^2(K^g)=\bigotimes\limits_{i=1}^g
L^2(K\tilde{q}_i)
\]
On each factor, we repeat the construction 
of the previous subsection and obtain an operator
\[
\Lambda\colon \bigotimes\limits_{i=1}^g L^2(K\tilde{q}_i)^{\absolute{G_i}}
\to
\bigotimes\limits_{i=1}^g L^2(K\tilde{q}_i)^{\absolute{G_i}}
\]
where $G_i$ is a verticial reduction graph of $T_i$. This operator generalises the product space operator from
\cite{BW2022} and decomposes as
\[
\Lambda_G=\Lambda_1+\dots+\Lambda_g
\]
where 
\[
\Lambda_i=1\otimes\cdots\otimes1\otimes\Lambda_{G_i}^{\Delta_i}\otimes 1\otimes\dots\otimes 1
\]
for $i=1,\dots,g$.

\smallskip
Instead of playing Game \ref{game} for each component graph $G_i$, we recover the product graph for the torus $A$:

\begin{cor}
Playing Game \ref{game} is possible for $p$-adic analytic tori in a successful way in order to recover the product graph composed of $G_1,\dots,G_g$.
\end{cor}

\begin{proof}
Let $G$ be the product graph of $G_1,\dots,G_g$. Then 
$\Lambda_G$ can be viewed in fact as an operator
\[
\Lambda_G\colon L^2(K)^{\absolute{G}}\to
L^2(K)^{\absolute{G}}
\]
by taking suitable isomorphisms. Then play Game \ref{game} using assumption (\ref{assumption}). This recovers $G$.
\end{proof}

\section*{Acknowledgements}
Evgeny Zelenov is thanked for posing this problem to one of the authors. 
David Weisbart is thanked for valuable discussions.
This research is partially supported by
the Deutsche Forschungsgemeinschaft under project number 469999674

\bibliographystyle{plain}
\bibliography{biblio}

\end{document}